\documentclass[11pt]{article}
\usepackage[margin=1in]{geometry} 
\usepackage{amssymb,amsfonts,amsmath,bbm,mathrsfs,stmaryrd,mathtools,wasysym}
\usepackage{xcolor}
\usepackage{url}

\usepackage{extarrows}

\usepackage[shortlabels]{enumitem}
\usepackage{tensor}

\usepackage{xr}
\usepackage[T1]{fontenc}

\usepackage[colorlinks,
linkcolor=black!75!red,
citecolor=blue,
pdftitle={},
pdfproducer={pdfLaTeX},
pdfpagemode=None,
bookmarksopen=true,
bookmarksnumbered=true,
backref=page]{hyperref}

\usepackage{tikz}
\usetikzlibrary{arrows,calc,decorations.pathreplacing,decorations.markings,intersections,shapes.geometric,through,fit,shapes.symbols,positioning,decorations.pathmorphing}

\usepackage{braket}

\usepackage[amsmath,thmmarks,hyperref]{ntheorem}
\usepackage{cleveref}

\creflabelformat{enumi}{#2#1#3}

\crefname{section}{Section}{Sections}
\crefformat{section}{#2Section~#1#3} 
\Crefformat{section}{#2Section~#1#3} 

\crefname{subsection}{\S}{\S\S}
\AtBeginDocument{%
  \crefformat{subsection}{#2\S#1#3}%
  \Crefformat{subsection}{#2\S#1#3}%
}

\crefname{subsubsection}{\S}{\S\S}
\AtBeginDocument{%
  \crefformat{subsubsection}{#2\S#1#3}%
  \Crefformat{subsubsection}{#2\S#1#3}%
}

\theoremstyle{plain}

\newtheorem{lemma}{Lemma}[section]
\newtheorem{proposition}[lemma]{Proposition}
\newtheorem{corollary}[lemma]{Corollary}
\newtheorem{theorem}[lemma]{Theorem}

\theoremstyle{plain}
\theoremnumbering{Alph}
\newtheorem{theoremN}{Theorem}

\theoremstyle{plain}
\theorembodyfont{\upshape}
\theoremsymbol{\ensuremath{\blacklozenge}}

\newtheorem{definition}[lemma]{Definition}
\newtheorem{example}[lemma]{Example}

\newtheorem{remark}[lemma]{Remark}
\newtheorem{remarks}[lemma]{Remarks}

\crefname{definition}{definition}{definitions}
\crefformat{definition}{#2definition~#1#3} 
\Crefformat{definition}{#2Definition~#1#3} 

\crefname{ex}{example}{examples}
\crefformat{example}{#2example~#1#3} 
\Crefformat{example}{#2Example~#1#3} 

\crefname{exs}{example}{examples}
\crefformat{examples}{#2example~#1#3} 
\Crefformat{examples}{#2Example~#1#3} 

\crefname{remark}{remark}{remarks}
\crefformat{remark}{#2remark~#1#3} 
\Crefformat{remark}{#2Remark~#1#3} 

\crefname{remarks}{remark}{remarks}
\crefformat{remarks}{#2remark~#1#3} 
\Crefformat{remarks}{#2Remark~#1#3} 

\crefname{convention}{convention}{conventions}
\crefformat{convention}{#2convention~#1#3} 
\Crefformat{convention}{#2Convention~#1#3} 

\crefname{notation}{notation}{notations}
\crefformat{notation}{#2notation~#1#3} 
\Crefformat{notation}{#2Notation~#1#3} 

\crefname{table}{table}{tables}
\crefformat{table}{#2table~#1#3} 
\Crefformat{table}{#2Table~#1#3}

\crefname{lemma}{lemma}{lemmas}
\crefformat{lemma}{#2lemma~#1#3} 
\Crefformat{lemma}{#2Lemma~#1#3} 

\crefname{proposition}{proposition}{propositions}
\crefformat{proposition}{#2proposition~#1#3} 
\Crefformat{proposition}{#2Proposition~#1#3} 

\crefname{propositionN}{proposition}{propositions}
\crefformat{propositionN}{#2proposition~#1#3} 
\Crefformat{propositionN}{#2Proposition~#1#3} 

\crefname{corollary}{corollary}{corollaries}
\crefformat{corollary}{#2corollary~#1#3} 
\Crefformat{corollary}{#2Corollary~#1#3} 

\crefname{corollaryN}{corollary}{corollaries}
\crefformat{corollaryN}{#2corollary~#1#3} 
\Crefformat{corollaryN}{#2Corollary~#1#3} 

\crefname{theorem}{theorem}{theorems}
\crefformat{theorem}{#2theorem~#1#3} 
\Crefformat{theorem}{#2Theorem~#1#3} 

\crefname{theoremN}{theorem}{theorems}
\crefformat{theoremN}{#2theorem~#1#3} 
\Crefformat{theoremN}{#2Theorem~#1#3} 

\crefname{enumi}{}{}
\crefformat{enumi}{#2#1#3}
\Crefformat{enumi}{#2#1#3}

\crefname{assumption}{assumption}{Assumptions}
\crefformat{assumption}{#2assumption~#1#3} 
\Crefformat{assumption}{#2Assumption~#1#3} 

\crefname{construction}{construction}{Constructions}
\crefformat{construction}{#2construction~#1#3} 
\Crefformat{construction}{#2Construction~#1#3} 

\crefname{equation}{}{}
\crefformat{equation}{(#2#1#3)} 
\Crefformat{equation}{(#2#1#3)}

\numberwithin{equation}{section}

\theoremstyle{nonumberplain}
\theoremsymbol{\ensuremath{\blacksquare}}

\newtheorem{proof}{Proof}
\newcommand\pf[1]{\newtheorem{#1}{Proof of \Cref{#1}}}

\newcommand\bC{{\mathbb C}}

\newcommand\bR{{\mathbb R}}
\newcommand\bS{{\mathbb S}}

\newcommand\bZ{{\mathbb Z}}

\newcommand\cC{{\mathcal C}}

\newcommand\cE{{\mathcal E}}
\newcommand\cF{{\mathcal F}}
\newcommand\cG{{\mathcal G}}

\newcommand\cP{{\mathcal P}}
\newcommand\cQ{{\mathcal Q}}

\newcommand\cS{{\mathcal S}}
\newcommand\cT{{\mathcal T}}
\newcommand\cU{{\mathcal U}}

\newcommand\ol{\overline}

\DeclareMathOperator{\id}{id}

\DeclareMathOperator{\im}{\mathrm{im}}

\DeclareMathOperator{\spn}{span}

\DeclareMathOperator{\wo}{\widehat{\otimes}}

\newcommand{\cat}[1]{\textsc{#1}}

\newcommand{\qedhere}{\mbox{}\hfill\ensuremath{\blacksquare}}

\newcommand{\xrightarrowdbl}[2][]{%
  \xrightarrow[#1]{#2}\mathrel{\mkern-14mu}\rightarrow
}

\title{Symmetric monoidal categories of conveniently-constructible Banach bundles}
\author{Alexandru Chirvasitu}

\begin{document}

\date{}

\newcommand{\Addresses}{{
  \bigskip
  \footnotesize

  \textsc{Department of Mathematics, University at Buffalo}
  \par\nopagebreak
  \textsc{Buffalo, NY 14260-2900, USA}  
  \par\nopagebreak
  \textit{E-mail address}: \texttt{achirvas@buffalo.edu}


}}

\maketitle

\begin{abstract}
  We show that a continuously-normed Banach bundle $\mathcal{E}$ over a compact Hausdorff space $X$ whose space of sections is algebraically finitely-generated (f.g.) over $C(X)$ is locally trivial (and hence the section space is projective f.g over $C(X)$); this answers a question of I. Gogi\'c. As a preliminary we also provide sufficient conditions for a quotient bundle to be continuous phrased in terms of the Vietoris continuity of the unit-ball maps attached to the bundles. Related results include (a) the fact that the category of topologically f.g. continuous Banach bundles over $X$ is symmetric monoidal under the (fiber-wise-maximal) tensor product, (b) the full faithfulness of the global-section functor from topologically f.g. continuous bundles to $C(X)$-modules and (c) the consequent identification of the algebraically f.g. bundles as precisely the rigid objects in the aforementioned symmetric monoidal category. 
\end{abstract}


\vspace{.5cm}


\tableofcontents

\section*{Introduction}

This is a follow-up of sorts to \cite{2405.14518v1}, on and around continuous Banach bundles $\cE\xrightarrowdbl{\pi}X$ (in the sense of \cite[Definition 13.4]{fd_bdl-1}) over compact Hausdorff base spaces whose spaces $\Gamma(\cE)$ of sections are {\it topologically} finitely-generated (f.g. for short) over the algebra $C(X)$ of continuous complex-valued functions on $X$: there are finitely many elements of $\Gamma(\cE)$ with dense $C(X)$-span.

The earlier paper sought to characterize topological finite generation in bundle-theoretic terms: the condition is equivalent to
\begin{itemize}[wide]
\item the {\it subhomogeneity} of the bundle (i.e. the boundedness of the fiber dimensions $\dim\cE_x$, $x\in X$);

\item the trivializability of the restrictions $\cE|_{X_d}$ by finite open covers for the {\it strata}
  \begin{equation*}
    X_d:=\left\{x\in X\ |\ \dim \cE_x=d\right\}
  \end{equation*}
  or: the locally trivial bundles $\cE|_{X_d}$ are {\it of finite type} \cite[Definition 3.5.7]{hus_fib};

\item and the requirement that the closed subsets
  \begin{equation*}
    X_{\le d}:=\left\{x\in X\ |\ \dim \cE_x\le d\right\}\subseteq X
  \end{equation*}
  be {\it $G_{\delta}$}, i.e. \cite[Problem 3H]{wil_top} countable intersections of opens. 
\end{itemize}

The {\it constructibility} of the paper's title refers precisely to this stratification of $X$ by $X_d$ carrying finite local trivializations for $\cE$. The term is very familiar in the sheaf-theoretic literature, where {\it $\cS$-constructible sheaves} (for a {\it stratification} $\cS$, which we can here treat loosely as meaning roughly the sort of partition $X=\coprod_d X_d$ provides) are, again roughly, those which resemble locally trivial bundles over each stratum: \cite[\S 2.2.10]{bbd}, \cite[Definition 4.5.3]{htt_d} and \cite[Definition 8.1.3]{ks_shv-mfld}, say, are variants of the general notion. 

In very much the same circle of ideas, the initial motivation for the material below was \cite[Problem 3.12]{gog_top-fg}: given
\begin{itemize}[wide]
\item the category equivalence between locally trivial finite-rank vector bundles over (compact Hausdorff) $X$ and that of projective f.g. $C(X)$-modules (the celebrated {\it (Serre-)Swan Theorem} of \cite[Theorem I.6.18]{kar_k}; originally \cite[Theorem 2]{zbMATH03179258});

\item and the fact \cite[Corollary 15.4.8.]{wo} that algebraically f.g. {\it Hilbert} modules (in the sense of \cite[Definition 15.1.5]{wo}) over a unital $C^*$-algebra are automatically projective
\end{itemize}
the cited \cite[Problem 3.12]{gog_top-fg} asks whether, in the same fashion, a continuous Banach bundle with {\it algebraically} (as opposed to only topologically) f.g. section space $\Gamma(\cE)$ (this time only a {\it Banach} rather than Hilbert module over $C(X)$) is locally trivial. \Cref{th:loctrivfg} gives the affirmative answer. 

\begin{theoremN}\label{thn:loctrivfg}
  Let $\cE\xrightarrowdbl{\pi}X$ be a continuous Banach bundle over a compact Hausdorff space.

  $\Gamma(\cE)$ is f.g. as a $C(X)$-module precisely when
  \begin{itemize}[wide]
  \item $\cE$ is locally trivial of finite rank;

  \item or, equivalently, $\Gamma(\cE)$ is projective f.g. as a $C(X)$-module.  \qedhere
  \end{itemize}  
\end{theoremN}

The discussion branches out in a number of directions, in part in order to develop some background necessary for \Cref{th:loctrivfg} itself and in part as a natural outgrowth therefrom. One useful tool in simplifying the bundles one works with, for instance, is that of passing to a {\it quotient} $\cE/\cF$ \cite[\S 9]{gierz_bdls} by a {\it subbundle} $\cF\le \cE$ \cite[\S 17, Exercise 41]{fd_bdl-1}. That requires knowledge of when such a quotient again has continuous (rather than only upper semicontinuous) norm. \Cref{th:ballsvietcont} gives such a continuity criterion (along the same lines as \cite[Proposition 5.7]{laz_bb-sel}), whose \Cref{cor:quotbyloctriv} we record here as a sample:

\begin{theoremN}\label{thn:quotbyloctriv}
  A quotient of a continuous Banach bundle over a compact Hausdorff space by a finite-rank locally trivial subbundle is again continuous.  \qedhere
\end{theoremN}

\Cref{se:tens} is a spin-off from an earlier attempt to prove \Cref{th:loctrivfg} by means of tensor-product constructions: assuming for simplicity that $\cE$ has only two fiber dimensions $d_0<d_1$, one would like to reduce the discussion to fiber dimensions $0$ and $1$ instead by substituting the {\it $d_1^{st}$ exterior power} $\bigwedge^{d_1}\cE$. This will annihilate fibers over $X_{d_0}$ and render those over $X_{d_1}$ 1-dimensional, as desired. The issue is to formalize the intuition that the fiber-wise operation $\bigwedge^{d_1}$ makes sense and behaves as one expects under appropriate finite generation conditions (e.g. bundle-theoretic tensor operations match their module and/or Banach-module counterparts). \Cref{th:tensidentify,th:algeqtop,th:autocont} and \Cref{cor:eotffg,cor:symmmon,cor:schurfunct,cor:tfgfullinmod} all touch on aspects of this problem, and we compress them into a preview.

\begin{theoremN}\label{thn:bunsymmoncat}
  Let $X$ be a compact Hausdorff space and $\tensor*[^{\cat{tfg}}^{(F)}]{\cat{Bun}}{_X^{\infty}}$ the category of topologically f.g. continuous Banach bundles thereon, with continuous, fiber-wise linear morphisms.

  \begin{enumerate}[(1),wide]
  \item $\tensor*[^{\cat{tfg}}^{(F)}]{\cat{Bun}}{_X^{\infty}}$ is closed under the (fiber-wise Banach-maximal) tensor product $\wo_X$ ($\astrosun_X$ in \cite[post Corollary 1.10]{zbMATH04134853}), as is the subcategory
    \begin{equation*}
      \tensor*[^{\cat{fg}}^{(F)}]{\cat{Bun}}{_X^{\infty}}
      \subset
      \tensor*[^{\cat{tfg}}^{(F)}]{\cat{Bun}}{_X^{\infty}}
    \end{equation*}
    of algebraically f.g. bundles. 

  \item Regarding the two categories as symmetric monoidal under that tensor product, The smaller consists of precisely the {\it rigid} \cite[Definition 2.10.11]{egno} objects in the larger. 

  \item The global-section functor $\Gamma$ restricts on $\tensor*[^{\cat{tfg}}^{(F)}]{\cat{Bun}}{_X^{\infty}}$ to a fully faithful symmetric monoidal embedding into both
    \begin{itemize}[wide]
    \item the category of unital Banach $C(X)$-modules with continuous module morphisms, symmetric monoidal under the maximal tensor product $\wo_{C(X)}$ \cite[\S III.3.8]{clm_ban-mod};

    \item and that of plain algebraic unital $C(X)$, with the algebraic tensor product $\otimes_{C(X)}$.  \qedhere
    \end{itemize}
  \end{enumerate}
\end{theoremN}

\subsection*{Acknowledgements}

I am grateful for illuminating exchanges with I. Gogi{\'c} and A. J. Lazar. This work was partially supported by NSF grant DMS-2001128. 

\section{Preliminaries}\label{se:prel}

Virtually everything below goes through fine over the reals, but we assume Banach spaces (along with algebras, bundles and so on) complex for definiteness. 

\subsection{Bundles}\label{subse:prel:bdl}

For general background on Banach bundles we refer variously to \cite{dg_ban-bdl,fd_bdl-1,gierz_bdls} and others, with more precise citations where needed. Without further specification, the phrase refers to the {\it (H) Banach bundles} (for `Hofmann' \cite[\S 3]{zbMATH03539905}) of \cite[Definition 1.1 and subsequent discussion]{dg_ban-bdl}:
\begin{itemize}[wide]
\item A continuous open surjection $\cE\xrightarrowdbl{\pi}X$ with Banach-space {\it fibers} (or {\it stalks}) $\cE_x:=\pi^{-1}(x)$, $x\in X$;

\item With continuous scalar multiplication and addition;

\item and {\it upper semicontinuous} norm (in the sense of \cite[Problem 7K]{wil_top}: preimages of $(\infty,t)$, $t\in \bR$ are open);

\item and such that
  \begin{equation*}
    V_{<\varepsilon}(U\mid 0\mid 0_x)
    :=
    \left\{s\in \cE\ |\ \pi(s)\in U\text{ and }\|s\|<\varepsilon\right\}
    ,\quad
    \text{neighborhood }U\ni x\in X
    ,\quad
    \varepsilon>0
  \end{equation*}
  form a fundamental system of neighborhoods around the trivial element $0_x\in \cE_x$.   
\end{itemize}

The notation $V_{\bullet}(\bullet)$ fits into the broader pattern
\begin{equation}\label{eq:vmidmid}
  V_{<\varepsilon}(U\mid \sigma\mid p)
  :=
  \left\{s\in \cE\ |\ \pi(s)\in U\text{ and }\|s-\sigma(\pi(p))\|<\varepsilon\right\}
\end{equation}
for a {\it section} $X\xrightarrow{\sigma}\cE$ of the bundle (i.e. \cite[p.9]{dg_ban-bdl} a continuous right inverse to $\pi$), a point $p\in \cE$ and an open neighborhood $U\ni \pi(p)$. Spaces of sections of interest in the sequel are $\Gamma(\cE)$ (all), $\Gamma_b(\cE)$ (bounded) and $\Gamma_0(\cE)$ (vanishing at infinity, for locally compact base spaces $X$). These all coincide when $X$ is compact Hausdorff (which will mostly be the case). 

Banach bundles with {\it continuous} norm are either themselves termed {\it continuous} or, for variety and following \cite[p.8]{dg_ban-bdl}, {\it (F)} (for `Fell': \cite[Definition preceding Proposition 1.1]{fell_ext}, \cite[Definition 13.4]{fd_bdl-1}, etc.). 

The category $\tensor*[]{\cat{Bun}}{_X}$ of (H) Banach bundles over $X$ has as morphisms continuous maps intertwining the projections to the base space $X$ and restricting fiber-wise to linear {\it contractions} (i.e. norm $\le 1$; as in, say, \cite[\S 1.2]{hk_shv-bdl}); additional decorations indicate variants: $\tensor*[^{(F)}]{\cat{Bun}}{_X}$ is the category of continuous bundles, $\tensor*[^{(F)}]{\cat{Bun}}{_X^{\infty}}$ allows maps that are fiber-wise linear but not necessarily contractive, and so forth. 

\subsection{Modules}\label{subse:prel:mod}

{\it Banach modules} over a Banach algebra $A$ (mostly $C(X)$ for compact Hausdorff $X$) are as in \cite[\S 2]{dg_ban-bdl} (and \cite[\S III.1.6]{clm_ban-mod}, \cite[Definition 2.6.1 plus assumption (2.6.1)]{dales_autocont}, \cite[\S 2.1]{hk_shv-bdl} etc.): the multiplication map
\begin{equation*}
  \text{{\it projective} tensor product \cite[\S II.1.6]{clm_ban-mod}}
  =:
  A\wo E
  \ni
  a\otimes v
  \xmapsto{\quad}
  av
  \in
  E
\end{equation*}
(for left modules, say) is contractive (weakly, i.e. of norm $\le 1$). We write $\cat{Ban}$ for the category of Banach spaces and (linear) contractions (the $\cat{Ban}_1$ of \cite[\S I.1.2]{clm_ban-mod}) and similarly, $\tensor*[_A]{\cat{Ban}}{}$ is the category of {\it non-degenerate} (or {\it essential} \cite[Definition III.1.17]{clm_ban-mod}, \cite[Definition 2.6.1]{dales_autocont}) left Banach $A$-modules with contractions as morphisms: those for which
\begin{equation*}
  \overline{\spn\left\{av\ |\ a\in A,\ v\in E\right\}}
  \overline{AE}
  =
  E
\end{equation*}
The notation replicates obviously to right modules. This is equivalent \cite[discussion post Definition 2.6.1]{dales_autocont} to the module being unital when $A$ is (as it mostly will). 

An additional right-hand `$\infty$' superscript indicates we consider arbitrary linear continuous morphisms: $\tensor*[]{\cat{Ban}}{^{\infty}}$ (the $\cat{Ban}_{\infty}$ of \cite[\S I.1.1]{clm_ban-mod}), $\tensor*[_A]{\cat{Ban}}{^{\infty}}$, $\tensor*[]{\cat{Ban}}{_A^{\infty}}$, etc. These are less well-behaved than the corresponding contractive versions as plain categories, but are on the other hand {\it enriched} over $\cat{Ban}$ (or {\it $\cat{Ban}$-categories} \cite[\S 1.2]{kly}, \cite[Definition 6.2.1]{brcx_hndbk-2}, etc.): spaces of morphisms are themselves Banach spaces and identities and compositions are morphisms in $\cat{Ban}$. Categories of plain unital (left) modules will be denoted by $\tensor*[_A]{\cat{Mod}}{}$ instead, and similarly for right modules and bimodules ($\tensor[_A]{\cat{Ban}}{_B}$, $\tensor[_A]{\cat{Mod}}{_B}$ and so on). 

\begin{remarks}\label{res:getbackban}
  \begin{enumerate}[(1),wide]
  \item The monoidal unit ${\bf 1}_{\cat{Ban}}$ of $\cat{Ban}$ is $\bC$, so the underlying set $\cat{Ban}\left({\bf 1}_{\cat{Ban},E}\right)$ of a Banach space is its unit ball. For that reason, one can recover the plain $\tensor*[_A]{\cat{Ban}}{}$ as the ordinary category {\it underlying} \cite[Corollary 6.4.4]{brcx_hndbk-2} the $\cat{Ban}$-category $\tensor*[_A]{\cat{Ban}}{^{\infty}}$.

  \item Suppose the Banach algebra $A$ has a {\it (left) bounded approximate unit} in the sense of \cite[Definition 28.51]{hr-2}: a net
    \begin{equation*}
      (u_\lambda)_{\lambda}
      ,\quad
      \left\{\|u_{\lambda}\|\right\}_{\lambda}
      \text{ bounded}
      ,\quad
      u_{\lambda}a
      \xrightarrow[\quad\lambda\quad]{}
      a
      ,\quad
      \forall a\in A.
    \end{equation*}
    By {\it Cohen factorization} (\cite[Theorem 32.22]{hr-2} or \cite[Theorem III.1.16]{clm_ban-mod}) $A$ is an {\it idempotent ring} \cite[\S 1, $1^{st}$ paragraph]{qln_nonunit}: $A^2=A$. There is then a well-behaved category of {\it firm} left $A$-modules, i.e. \cite[Definition 2.3]{qln_nonunit} those for which the canonical map
    \begin{equation}\label{eq:aam2m}
      A\otimes_AM
      \xrightarrow{\quad}
      M
    \end{equation}
    is an isomorphism. That would be the reasonable purely-algebraic analogue of the category $\tensor*[_A]{\cat{Ban}}{}$ of non-degenerate Banach $A$-modules. Indeed, \cite[Proposition II.3.13]{hlmsk_homolog} (or \cite[Theorem III.3.10]{clm_ban-mod}) shows that for Banach modules non-degeneracy is equivalent to the analytic counterpart to \Cref{eq:aam2m}: the requirement that the canonical map $A\wo_AM\to M$ be an isomorphism. 
  \end{enumerate}
\end{remarks}

\section{Finitely-generated section spaces and local triviality}\label{se:fgloctriv}

\Cref{th:loctrivfg} is meant to address \cite[Problem 3.12]{gog_top-fg}. In the proof we will make frequent use of the following notation for a Banach bundle $\cE\xrightarrowdbl{}X$:
\begin{equation*}
  X_{\text{numerical condition}}
  :=
  \left\{x\in X\ |\ \dim\cE_x\text{ meet said condition}\right\}.
\end{equation*}
Examples include $X_d$, $d\in \bZ_{\ge 0}$ for the locus where the fibers have dimension precisely $d$, $X_{>d}$ for strictly larger dimensions, etc. We occasionally decorate the symbol with that of the bundle as well, for clarity, always in self-explanatory fashion: $X_{d}$ might be $X_{\cE=d}$, $X_{>d}=X_{\cE>d}$, and so on.

For continuous $\cE\xrightarrowdbl{}X$ The {\it strata} $X_d$ are in any case {\it locally closed} (open in their closure, or intersections of closed and open sets \cite[\S I.3.3, Definition 2 and Proposition 5]{bourb_top_en_1}) because $X_{>d}$ are open (countable unions of closed sets). 

A few attributes of interest that Banach bundles may or may not have are as follows (see e.g. \cite[discussion surrounding Proposition 2.1]{gog_top-fg} for a similar quick convenient recollection).

\begin{definition}\label{def:bdleattrs}
  A bundle $\cE\xrightarrowdbl{} X$ over locally compact Hausdorff $X$ is  
  \begin{itemize}[wide]
  \item {\it homogeneous} \cite[Introduction]{dupre_hilbund-1} if its fibers all have the same constant dimension;

  \item {\it subhomogeneous} if there is a finite upper bound on their fiber dimensions $\dim\cE_x$;

  \item {\it (locally) trivial} (\cite[Definition 17.1]{gierz_bdls}, \cite[Introduction]{dupre_hilbund-1}) if $\cE\cong X\times \left(\text{some Banach space }E\right)$ (respectively, such isomorphisms hold over open patches covering $X$);    

  \item {\it conditionally} (locally) trivial, (always under the assumption of subhomogeneity) if the restrictions of $\cE$ to the {\it strata}
    \begin{equation}\label{eq:strata}
      X_{d_i}
      ,\quad
      d_0<\cdots<d_{k-1} 
    \end{equation}
    are (locally) trivial;

  \item {\it of finite type (f.t.)} \cite[Definition 3.5.7]{hus_fib} if it is trivialized by a finite open cover;

  \item {\it conditionally of finite type (f.t.)} (always assuming subhomogeneity) if the restrictions to the strata are such;

    
  \item {\it (topologically) finitely-generated} ({\it f.g.} for short) if $\Gamma_0(\cE)$ is so as a Banach module over $C_0(X)$ (continuous functions on $X$ vanishing at infinity). 
  \end{itemize}
\end{definition}

Recalling \cite[\S 40, Exercise 2]{mnk} that the {\it $G_{\delta}$ subsets} of a topological space are the countable intersections of open sets, \cite[Theorem 1.11]{2405.14518v1} shows that for (F) bundles over compact Hausdorff bases $X$
\begin{equation*}
  \text{topological finite generation}
  \iff
  \text{aggregate}
  \begin{cases}
    \text{subhomogeneity}\\
    \text{conditional f.t.}\\
    \text{$X_{\le d}$ are $G_{\delta}$}
  \end{cases}
\end{equation*}

As another helpful piece of notation, when discussing a problem $\cP$, say, for subhomogeneous Banach bundles, we write
\begin{equation*}
  \cP_{\bf d}
  ,\quad
  {\bf d} = (d_0,\cdots,d_k) =  (d_0<\cdots<d_{k-1})
\end{equation*}
for instances of that problem for bundles whose strata carry precisely the fiber dimensions listed by the tuple ${\bf d}$. 

\begin{remark}\label{re:ballnot}
  The reader should not confuse symbols such as $X_{\le d}$ with the similar ones we use for closed or open balls in Banach spaces: $E_{\le r}$ and $E_{<r}$ are the closed and open balls of radius $r$ respectively in the Banach space $E$. 
\end{remark}

\begin{theorem}\label{th:loctrivfg}
  For an (F) Banach bundle $\cE\xrightarrowdbl{}X$ over a compact Hausdorff space the following conditions are equivalent.
  \begin{enumerate}[(a)]
    
  \item\label{item:th:loctrivfg:loctriv} $\cE$ is locally trivial of finite rank.

  \item\label{item:th:loctrivfg:pfg} The Banach $C(X)$-module $\Gamma(\cE)$ of global sections is algebraically finitely-generated and projective.
    
  \item\label{item:th:loctrivfg:fg} $\Gamma(\cE)$ is algebraically finitely-generated.

  \end{enumerate}  
\end{theorem}

It will be convenient to simplify the bundles under consideration by passing to {\it quotient bundles} \cite[\S 2.4]{kqw_rieff} thereof. Recall first the notion of a {\it subbundle} $\cF\le \cE$ (\cite[\S II.17, Exercise 41]{fd_bdl-1} for (F) and \cite[Definition 2.9]{kqw_rieff} for the broader class of (H) bundles): a collection of closed subspaces $\cF_x\le \cE_x$ for $x\in X$ which aggregate into an (H) Banach bundle in its own right with the topology projection map to $X$ inherited from $\cE$. The disjoint union
\begin{equation*}
  \cE/\cF
  :=
  \coprod_{x\in X}\cE_x/\cF_x,
\end{equation*}
with its quotient topology, has a natural (H) bundle structure \cite[Proposition 2.16]{kqw_rieff} (see also \cite[\S 1]{laz_bb-sel}).

Over compact Hausdorff spaces (H) Banach bundles admit \cite[Scholium 6.7]{hk_shv-bdl} that  a number of mutually equivalent sheaf-theoretic interpretations:
\begin{itemize}[wide]
\item as sheaves of Banach modules over the sheaf $\cC=\cC(X)$ of bounded continuous functions on $X$;

\item as sheaves of Banach $C(X)$-modules satisfying an additional constraint of being {\it well-supported} \cite[Definition 4.2]{hk_shv-bdl};

\item as the {\it Banach sheaves} of \cite[\S 3]{MR0560785} (or {\it approximation sheaves} of \cite[\S 3.5]{hk_shv-bdl}). 
\end{itemize}
As a sheaf, the quotient bundle just recalled is nothing but (the bundle corresponding to) the quotient sheaf in the familiar sense (\cite[Definition I.2.4]{bred_shf_2e_1997}, \cite[\S II.1]{hrt}, etc.).

The theory in those sources does not {\it quite} apply directly; it concerns sheaves valued in the category $\cat{Set}$ of sets rather than $\cat{Ban}$, which does make a difference: stalks (fibers) $\cE_x$ are computed as colimits in the two respective categories, the $\cat{Set}$-analogue $\cE\xrightarrow{\pi}X$ of a Banach bundle map ($\cE$ being the {\it \'etal\'e space} \cite[Exercise II.1.13]{hrt} of the sheaf) is a local homeomorphism, etc. Much of the theory, though, replicates fairly straightforwardly. 

\begin{lemma}\label{le:quotshvs}
  Let $\cF\le \cE$ be an embedding of Banach bundles over a compact Hausdorff space $X$.

  The sheaf associated to the quotient $\cE/\cF$ is the Banach sheaf $\cat{BanSh}(\cat{qpsh}_{\cE/\cF})$ attached to the presheaf
  \begin{equation*}
    \left(\text{open }U\subseteq X\right)
    \xmapsto{\quad\cat{qpsh}_{\cE/\cF}\text{ (for `quotient presheaf')}\quad}
    \Gamma_b(\cE|_U)/\Gamma_b(\cF|_U)
  \end{equation*}
  as the latter's image through the left adjoint $\cat{BanSh}$ to the inclusion functor
  \begin{equation*}
    \cat{BanSh}_X:=\text{Banach sheaves on $X$}
    \lhook\joinrel\xrightarrow{\quad}
    \cat{BanPSh}_X:=\text{Banach-space presheaves on $X$}.
  \end{equation*}
\end{lemma}
\begin{proof}
  This will be immediate upon unpacking the requisite universal properties in the two settings (bundles and sheaves). By the very definition of $\cQ:=\cE/\cF$ as the disjoint union of quotients $\cE_x/\cF_x$, $x\in X$ equipped with the quotient topology,
  \begin{equation*}
    \cat{Bun}_X(\cQ,\cG)
    \cong
    \left\{\varphi\in \cat{Bun}_X(\cE,\cG)\ |\ \varphi|_{\cF}\equiv 0\right\}
  \end{equation*}
  functorially in $\cG\in \cat{Bun}_X$. Recasting bundles as sheaves, the latter space is also
  \begin{equation*}
    \begin{aligned}
      \left\{\varphi\in \cat{BanSh}_X(\cE,\cG)\ |\ \varphi|_{\cF}\equiv 0\right\}
      &\cong
        \cat{BanPSh}_X(\cat{qpsh}_{\cE/\cF},\cG)\\
      &\cong
        \cat{BanSh}_X(\cat{BanSh}(\cat{qpsh}_{\cE/\cF}),\cG).
    \end{aligned}
  \end{equation*}
  The two objects we wish to identify,
  \begin{equation*}
    \cQ\in \cat{Bun}_X
    \quad
    \simeq
    \quad
    \cat{BanSh}_X
    \ni
    \cat{BanSh}(\cat{qpsh}_{\cE/\cF}),
  \end{equation*}
  thus represent isomorphic functors and must themselves be isomorphic. 
\end{proof}

\begin{remarks}\label{res:quotnotf}
  Assume the base space $X$ compact Hausdorff throughout. 
  \begin{enumerate}[(1),wide]
    
  \item\label{item:res:quotnotf:quotnotf} Quotient bundles $\cE/\cF$ certainly need not be (F), even when $\cE$ and $\cF$ both are: the quotient $\cQ$ of the trivial rank-1 bundle on $X$ by the bundle corresponding to the ideal $C_0(X\setminus\{x\})$ has rank 1 at $x$ and 0 elsewhere, and hence will be (F) precisely when $x\in X$ is isolated (i.e. when $\{x\}=X_{\cQ>0}$ is open). \cite[Example 5.5]{laz_bb-sel} is an instance of this observation (with $X$ consisting of a convergent sequence together with its limit $x$). 
    
  \item\label{item:res:quotnotf:sectquot} The surjection $\cE\xrightarrowdbl{} \cE/\cF$ always induces one at the level of sections: $\Gamma(\cE)\xrightarrowdbl{}\Gamma(\cE/\cF)$. This is a consequence of \cite[Theorem 3.2]{laz_bb-sel}, itself an application of the (Michael-style \cite[Theorem 3.2'']{mich_contsel-1}, \cite[Theorem 1.1]{zbMATH06329568}) {\it selection theorem} \cite[Theorem 2.9]{laz_bb-sel}. The hypothesis requires the openness of the map $\cE\to \cE/\cF$, which does indeed obtain (\cite[\S 9.4]{gierz_bdls} or \cite[Proposition 2.16]{kqw_rieff}).
   
  \item\label{item:res:quotnotf:nocohom} Alternatively, the surjectivity of $\Gamma(\cE)\xrightarrowdbl{} \Gamma(\cE/\cF)$ has a sheaf-theoretic interpretation: sheaves corresponding to (H) Banach bundles are, by \cite[p.15, pre Proposition 1.3]{dg_ban-bdl} (also \cite[Lemma 3.3]{zbMATH03663842}), {\it soft} in the sense of \cite[Definition II.9.1]{bred_shf_2e_1997}: sections over closed subsets extend globally. It follows \cite[Theorem II.9.9]{bred_shf_2e_1997} that applying the global-section functor to
    \begin{equation*}
      0\to
      \cF
      \lhook\joinrel\xrightarrow{\quad}
      \cE
      \xrightarrowdbl{\quad}
      \cE/\cF
      \to 0
    \end{equation*}
    interpreted as a short exact sequence of sheaves produces a short exact sequence again.

    The caveat preceding \Cref{le:quotshvs} applies: \cite[Theorem II.9.9]{bred_shf_2e_1997} requires slight adjustments to the $\cat{Ban}$ (rather than $\cat{Set}$) setup. The proof transports verbatim so long as {\it supports} \cite[Definition I.1.10]{bred_shf_2e_1997} of sections $\sigma\in \Gamma(\cE)$ are defined as
    \begin{equation*}
      |\sigma|
      :=
      \text{closure }\ol{\left\{x\in X\ |\ \sigma(x)\ne 0\right\}}
      \subseteq X.
    \end{equation*}
    For sheaves valued in (discrete) abelian groups taking the closure is redundant, for the non-zero locus is automatically closed.
  \end{enumerate}  
\end{remarks}

The quotient in \Cref{res:quotnotf}\Cref{item:res:quotnotf:quotnotf} fails to be (F) for very simple numerical reasons: for (F) Banach bundles the fiber-dimension function
\begin{equation}\label{eq:dimfn}
  X\ni x
  \xmapsto{\quad}
  \dim\cE_x
  \in \bZ_{\ge 0} 
\end{equation}
is lower semicontinuous. That condition fails in said example, but even it is no guarantee that quotients of (F) bundles are (F). 


\begin{example}\label{ex:semicontstillnotf}
  Fix
  \begin{itemize}[wide]
  \item a point $x\in X$ in a compact Hausdorff space;
  \item a direct-sum decomposition $E=F\oplus F'$ of finite-dimensional Banach spaces (so that in particular $\dim E-\dim F=\dim F'$);
  \item and a vector $v\in F'$ whose distance from $F$ is strictly smaller than $\|v\|$. 
  \end{itemize}
  Take for $\cE$ the bundle corresponding to the $C(X)$-module
  \begin{equation*}
    \Gamma(\cE)
    :=
    \left\{X\xrightarrow[\text{continuous}]{f}E\ |\ f(x)\in F'\right\}
  \end{equation*}
  (a subbundle of the trivial $X\times E$) and similarly,
  \begin{equation*}
    \Gamma(\cF)
    :=
    \left\{X\xrightarrow[\text{continuous}]{f}F\ |\ f(x)=0\right\}.
  \end{equation*}
  The quotient $\cE/\cF$ has constant fiber dimension $\dim E-\dim F=\dim F'$, but the ``constant'' section taking the value $v\in F'$ in every fiber has discontinuous norm at the highlighted point $x$: taking the value $\|v\|$ there, but $d(v,F)<\|v\|$ elsewhere. 
\end{example}

\begin{remark}\label{re:quott2notcont}
  \Cref{ex:semicontstillnotf} will also serve the purpose of answering \cite[Question 5.6]{laz_bb-sel} in the negative: the quotient bundle $\cQ:=\cE/\cF$ produced there is (as observed) not continuous, but its total space is nevertheless Hausdorff. To see this, simply note the following alternative description of $\cQ$: its total space is (topologically and linearly along each fiber) simply $X\times F'$, with the fibers $\cQ_{x'}\cong F'$ equipped with
  \begin{itemize}[wide]
  \item the original norm on $F'$ regarded as a subspace of $E$ at $x'=x$;
  \item and the equivalent but different norm pulled back through
    \begin{equation*}
      \begin{tikzpicture}[>=stealth,auto,baseline=(current  bounding  box.center)]
        \path[anchor=base] 
        (0,0) node (l) {$F'$}
        +(2,.5) node (u) {$E$}
        +(4,0) node (r) {$E/F$}
        ;
        \draw[right hook->] (l) to[bend left=6] node[pos=.5,auto] {$\scriptstyle $} (u);
        \draw[->>] (u) to[bend left=6] node[pos=.5,auto] {$\scriptstyle $} (r);
        \draw[->] (l) to[bend right=6] node[pos=.5,auto,swap] {$\scriptstyle \cong$} (r);
      \end{tikzpicture}
    \end{equation*}
    elsewhere. 
  \end{itemize}
\end{remark}


\Cref{th:ballsvietcont} below is very much in the spirit of \cite[Proposition 5.7]{laz_bb-sel}, also providing sufficient conditions for the continuity of the norm of a quotient Banach bundle. That statement does not apply directly, as it concerns {\it locally uniform} Banach bundles in the sense of \cite[duscission post Proposition 5.2]{laz_bb-sel}: those obtained, locally on the base $X$, by setting
\begin{equation*}
  \cE_x:=\psi(x)
  ,\quad
  \text{open $U\subseteq X$}
  \xrightarrow{\quad\psi\quad}
  \left(\text{closed subspaces of a Banach space $E$}\right)
\end{equation*}
with $\psi$ having the property that
\begin{equation*}
  X\ni x
  \xmapsto{\quad}
  \left(\text{unit ball of $\psi(x)$}\right)
\end{equation*}
is continuous for the {\it Hausdorff metric} \cite[Definition 7.3.1]{bbi} on the space of bounded closed subsets of $E$:
\begin{equation*}
  \begin{aligned}
    d(A,B)
    &:= \max\left(d(a,B),\ d(b,A)\ |\ a\in A,\ b\in B\right)\\
    d(a,B)
    &:= \inf\left\{d(a,b)\ |\ b\in B\right\}.
  \end{aligned}  
\end{equation*}
To state a version of that result appropriate outside the scope of locally uniform bundles, we need to recall some background on topologies on classes of subsets of a topological space. Specifically, following \cite[\S 1.1]{ct_vietoris} or \cite[\S 1.3]{kt_corresp}:

\begin{itemize}[wide]
  
\item The {\it lower Vietoris topology} $\cT_{V}^{-}$ on a class $\cC$ of subsets (typically closed) of a topological space $X$ has
  \begin{equation*}
    \left\{A\in \cC\ |\ A\cap U\ne \emptyset\right\}
    ,\quad
    U\subseteq X\text{ open}
  \end{equation*}
  as a subbase of open sets.
  
\item The {\it upper Vietoris topology} $\cT_{V}^{+}$ has
  \begin{equation*}
    \left\{A\in \cC\ |\ A\subset U\right\}
    ,\quad
    U\subseteq X\text{ open}
  \end{equation*}
  as a (sub)base of open sets.

\item The {\it (plain) Vietoris topology} $\cT_{V}$ is the supremum of $\cT_V^{\pm}$. 
\end{itemize}

\begin{theorem}\label{th:ballsvietcont}
  Let $\cF\le \cE$ be an embedding of (F) Banach bundles over compact Hausdorff $X$. The quotient $\cE/\cF$ is then again (F) provided the map
  \begin{equation}\label{eq:ballmap}
    X\ni x
    \xmapsto{\quad}
    \left(\text{closed $r$-ball }(\cF_x)_{\le r}\right)
    \in
    \mathrm{Cl}(\cE)
    :=
    \left\{\text{closed subsets of }\cE\right\}
  \end{equation}
  is upper-Vietoris continuous for some or equivalently all $r>0$.
\end{theorem}

Locally trivial finite-rank subbundles of (F) bundles plainly satisfy the hypothesis on $\cF$, hence:

\begin{corollary}\label{cor:quotbyloctriv}
  A quotient of an (F) Banach bundle over a compact Hausdorff space by a finite-rank locally trivial (hence also (F)) subbundle is again (F).  \qedhere
\end{corollary}

\pf{th:ballsvietcont}
\begin{th:ballsvietcont}
  The mutual equivalence of the two conditions is plain enough, as one radius $r$ will recover the others by scaling. 
  
  We denote by $\pi$ the various bundle projections to the base space $X$, relying on context to distinguish among them. One way to prove that the norm of $\cQ:=\cE/\cF$ is continuous is, by \cite[Theorem 2.6]{dg_ban-bdl} (say), to show that the function
  \begin{equation*}
    X\ni x
    \xmapsto{\quad}
    \|\sigma(x)\|_{\cQ_x}
  \end{equation*}
  is continuous at a fixed arbitrary $x_0\in X$ for a fixed but again arbitrary $\sigma\in \Gamma(\cQ)$. Furthermore, by the surjectivity of $\Gamma(\cE)\xrightarrowdbl{}\Gamma(\cQ)$ (\Cref{res:quotnotf}\Cref{item:res:quotnotf:sectquot}), we can assume that $\sigma$ is in fact a section of the ambient bundle $\cE$ (retaining the symbol slightly abusively). Finally, we are in fact only interested in the {\it lower} semicontinuity, for the upper comes for free from the general theory of (H) quotient bundles. 

  Suppose $\|\sigma(x_0)\|_{\cQ}>K\ge 0$. Extending the notation \Cref{eq:vmidmid}, the continuity of the norm of $\cE$ implies that
  \begin{equation}\label{eq:vusigma}
    V_{>K}(U\mid\sigma)
    =
    V_{>K}(U\mid\sigma\mid \sigma(x_0))
    :=
    \left\{s\in E\ |\ \pi(s)\in U\text{ and }\|\sigma(\pi(s))-s\|_{\cE_{\pi(p)}}>K\right\}
    \subset \cE
  \end{equation}
  is open for every open neighborhood $U\ni x_0$. Every ball $(\cF_{x_0})_{\le r}$ is contained in that open set by assumption, hence so are $(\cF_{x})_{\le r}$ for $x$ close to $x_0$ by upper-Vietoris continuity. Or: for every $r>0$ there is a neighborhood $V\ni x_0$ such that
  \begin{equation*}
    x\in V
    \xRightarrow{\quad}
    d(\sigma(x),(\cF_{x})_{\le r})>K.
  \end{equation*}
  For $x$ sufficiently close to $x_0$ there will be some $r$ such that
  \begin{equation*}
    d\left(x,\cF_x\right)\le K
    \xLeftrightarrow{\quad}
    d\left(x,(\cF_x)_{\le r}\right)\le K,
  \end{equation*}
  hence the conclusion.
\end{th:ballsvietcont}

\begin{remark}
  \cite[Proposition 5.7]{laz_bb-sel} does not quite follow from \Cref{th:ballsvietcont} directly, for the former's continuity assumption is, formally, weaker than the latter's: per \cite[Proposition 4.2.1(i)]{kt_corresp}, upper-Vietoris convergence entails convergence in the {\it upper (Hausdorff) hemimetric} \cite[post Proposition 4.1.4]{kt_corresp}
  \begin{equation*}
    d_{+}(A,B):=\inf\left\{\varepsilon\ge 0\ |\ B\subseteq\text{$\varepsilon$-neighborhood of }A\right\}.
  \end{equation*}
  but not conversely. The {\it proof} of \Cref{th:ballsvietcont}, though, does recover \cite[Proposition 5.7]{laz_bb-sel}: for locally uniform Banach bundles $\cF\le \cE$ over $X$ and $x_0\in X$, $\sigma$ and so on as in the proof of \Cref{th:ballsvietcont}, assuming
  \begin{equation*}
    (\cF_{x_0})_{\le r}
    \subset
    V_{>K}(U\mid\sigma)
    \quad
    \left(\text{in the notation of \Cref{eq:vusigma}}\right),
  \end{equation*}
  if $x$ is sufficiently close to $x_0$ (so that the ball $(\cF_x)_{\le r}$ is Hausdorff-close to $(\cF_{x_0})_{\le r}$) we have $(\cF_{x})_{\le r}\subset V_{>K-\varepsilon}(U\mid\sigma)$ for pre-selected small $\varepsilon>0$.

  It would be possible to state a common generalization of \cite[Proposition 5.7]{laz_bb-sel} and \Cref{th:ballsvietcont}, and the latter's proof essentially delivers it; the statement, though, would be a little awkward: one would have to require a formally weaker form of upper-Vietoris convergence, involving the sets $V_{>K}(U\mid\sigma)$ rather than arbitrary open subsets of $\cE$. 
\end{remark}


\pf{th:loctrivfg}
\begin{th:loctrivfg}
  \Cref{item:th:loctrivfg:loctriv} $\Rightarrow$ \Cref{item:th:loctrivfg:pfg} is a consequence of Swan's Theorem \cite[Theorem 2]{zbMATH03179258}, while \Cref{item:th:loctrivfg:pfg} $\Rightarrow$ \Cref{item:th:loctrivfg:fg} is formal, so the substantive implication is \Cref{item:th:loctrivfg:fg} $\Rightarrow$ \Cref{item:th:loctrivfg:loctriv}.
  
  The subhomogeneity of $\cE$ and the fact that every $X_d$ admits a finite $\cE$-trivializing open cover already follows \cite[Theorem 1.11]{2405.14518v1} from {\it topological} finite generation. The issue at hand is to show that the stronger {\it algebraic} assumption also implies the continuity of the dimension function \Cref{eq:dimfn} (as opposed to its lower semicontinuity). Equivalently, the claim is that the strata $X_d$, a priori {\it locally} closed, are in fact all closed (or all open, or all clopen). We simplify the setup progressively.

  \begin{enumerate}[(I),wide]

  \item\label{item:th:loctrivfg:condtriv} {\bf : reduction to conditionally trivial $\cE$.} Recall from \Cref{def:bdleattrs} that {\it Conditionally} trivial means trivial over every stratum $X_d$, $d\in \bZ_{\ge 0}$. \Cref{le:fincov} below shows (via \cite[Theorem 1.11]{2405.14518v1}, which ensures the hypothesis of the lemma is equivalent to topological finite generation) that $X$ has a finite closed cover $X=\bigcup_{i=1}^s Y_i$ with $\cE|_{Y_i}$ conditionally trivial. Every net
    \begin{equation*}
      X_{d}\supset (x_{\lambda})
      \xrightarrow[\lambda]{\quad}
      x\in X_{d'}
      ,\quad
      d>d'
    \end{equation*}
    will have a convergent subnet lying within a single $Y_i$, so we can restrict attention to the latter. 
    
  \item\label{item:th:loctrivfg:2str} {\bf : reduction to $\cP_{(d_0,d_1)}$.} This will be an induction on the number $k$ of strata $X_{d_i}$, $d_0<d_1<\cdots<d_{k-1}$. Assuming by the inductive hypothesis that the $(k-1)$-strata case is settled for $k\ge 3$, we have the clopen partition
    \begin{equation*}
      \text{closed }X_{<d_{k-1}} = X_{\le d_{k-1}-1} = \coprod_{i=0}^{k-2}X_{d_i}.
    \end{equation*}
    We can now restrict attention to small compact neighborhoods of points in individual $X_{d_i}$, $i<k-1$ to conclude (by the 2-strata case, assumed settled for the purpose of the present argument) that that the closure of $X_{d_{k-1}}$ does not meet said $X_{d_{i}}$, $i<k-1$. This completes the induction step.     

  \item\label{item:th:loctrivfg:d00} {\bf : reduction to $\cP_{(0,d_1)}$.} We are assuming $\cE$ conditionally trivial and hence its restriction $\cE_{d_0}:=\cE|_{X_{d_0}}$ trivial. Extend a nowhere-zero section $\sigma\in \Gamma(\cE_{d_0})$ to all of $X$ (denoting the extension by the same symbol), and replace $X$ with a closed neighborhood of $X_{d_0}$ over which $s$ is non-zero.

    Now,
    \begin{equation*}
      \cF:=\coprod_{x\in X}\cF_x
      ,\quad
      \cF_x:=\bC \sigma(x)
    \end{equation*}
    is a subbundle in the sense of \cite[\S II.17, Exercise 41]{fd_bdl-1}, so we can form the quotient bundle $\cQ:=\cE/\cF$, (F) rather than just (H) by \Cref{cor:quotbyloctriv}. The fiber dimensions of $\cQ$ are $d_0-1$ and $d_1-1$ and the strata are the same (i.e. $X_{\cQ=d_i-1}=X_{\cE=d_i}$, $i=0,1$), so we can repeat the procedure until the lower dimension vanishes. Passing to bundle quotients will not affect (either algebraic or topological) finite generation by \Cref{res:quotnotf}\Cref{item:res:quotnotf:sectquot}.
    
  \item\label{item:th:loctrivfg:01str} {\bf : reduction to $\cP_{0,1}$.} We are now considering the $\cP_{0,d_1}$ instance of the problem. We can proceed as in the previous step, this time quotienting instead by the subbundle spanned by a section vanishing {\it precisely} on $X_0$: such sections do exist by \cite[\S 33, Exercise 4]{mnk}, because \cite[Theorem 1.11]{2405.14518v1} $X_0\subseteq X$ is closed and $G_{\delta}$.

    The subbundle $\cF\le \cE$ is now (locally) trivial only over $X_1$, hence so is the quotient $\cE/\cF$. Continuity of the norm at points of $X_1$ then follows from the fact that all bundles vanish there. 
    
  \item {\bf : Conclusion.} After steps \Cref{item:th:loctrivfg:condtriv} and \Cref{item:th:loctrivfg:01str} we have $X=(Z:=X_0)\coprod (U:=X_1)$, $\cE|_U$ is trivial, and
    \begin{equation*}
      \Gamma(\cE)=
      C_0(U) = C_0(X\setminus Z) = \left\{f\in C(X)\ |\ f|_Z\equiv 0\right\}. 
    \end{equation*}
    The balance of the claim is now as follows: if, for an open subset $U\subseteq X$, the ideal $C_0(U)\subseteq C(X)$ is (algebraically) finitely-generated, then $U$ is also closed. Because, however, $C_0(U)$ is a {\it Hilbert} $C(X)$-module in the sense of \cite[Definition 15.1.5]{wo}, its algebraic finite generation makes it projective \cite[Corollary 15.4.8]{wo} and hence the section module of a locally trivial bundle by Serre-Swan again.
  \end{enumerate}
  This completes the proof.  
\end{th:loctrivfg}

The hypotheses on a Banach bundle $\cE\xrightarrowdbl{} X$ of the following result are jointly equivalent, by \cite[Theorem 1.11]{2405.14518v1}, to $\Gamma(\cE)$ being; stated in its present form, however, renders \Cref{le:fincov} independent of that earlier work. 

\begin{lemma}\label{le:fincov}
  Let $\cE\xrightarrowdbl{}X$ be a subhomogeneous (F) Banach bundle, conditionally f.t. and with $G_{\delta}$ $X_{\le d}$.
  
  $X$ has a finite closed cover $X=\bigcup_{i=1}^s Y_i$ with $\cE|_{Y_i}$ conditionally trivial for all $i$. 
\end{lemma}
\begin{proof}
  Stratify $X$ by the $X_d$, $d_0<\cdots<d_{k-1}$. We conditionally trivialize $\cE$ over progressively larger $X_{\le d}$, each time focusing on one member of a judiciously-chosen finite closed cover.

  First cover $X$ with finitely many closed sets over which $\cE|_{X_{d_0}}$ is trivial, together with closed sets avoiding $X_{d_0}$ and hence amenable to induction on $k$. In homing in on a single such set, we may as well assume $\cE|_{X_{d_0}}$ trivial. This constitutes the first step in an inductive process. To proceed, assume $\cE_{X_{\le d_i}}$ conditionally trivial. In proving the main claim for $X_{\le d_{i+1}}$, we will have completed the induction step. The simplified setup, then, is this:
  \begin{itemize}[wide]
  \item $\cE|_Z$ is conditionally trivial for $Z\subseteq X$ closed $G_{\delta}$ (having substituted the symbols $X$ and $Z$ for $X_{\le d_{i+1}}$ and $X_{\le d_i}$ respectively);
  \item and $\cE|_U$ is homogeneous of finite type for $U:=X\setminus Z$.
  \end{itemize}
  We are assuming that $\cE|_{\cU}$ is trivialized by a finite open cover $U=\bigcup_{i=1}^n U_n$, and also that $U$ is $F_{\sigma}$ in $X$ and hence (\cite[Theorems 20.7 and 20.10]{wil_top}) paracompact so also normal. It follows \cite[Theorem 36.1, Step 1 of the proof]{mnk} that there is an open cover of $U$ by $V_i$, $1\le i\le n$ with $V_i$'s closure in $U$ contained in $U_i$. The desired finite cover is by $\overline{V_i}$ (closure in $X$).
\end{proof}

\begin{remark}\label{re:notopencover}
  A finite cover as in \Cref{le:fincov}, conditionally trivializing $\cE$, can certainly not be chosen {\it open} in general: take for $X$ the {\it cone} \cite[p.9]{hatch_at}
  \begin{equation*}
    X
    :=
    CY
    :=
    Y\times [0,1]/Y\times \{0\}
  \end{equation*}
  on a compact Hausdorff space $Y$, and for $\cE$ a bundle restricting to
  \begin{equation*}
    \pi^*\cF
    ,\quad
    Y\times (0,1]
    \xrightarrow{\quad\pi\quad}
    Y
  \end{equation*}
  for non-trivial $\cF\xrightarrowdbl{}Y$ and trivial at the tip of the cone. $\cE$ cannot be trivialized in any neighborhood of that tip. 
\end{remark}

\Cref{le:fincov} allows the reduction of general problems concerning the bundles mentioned there to particularly pleasant section modules.

\begin{corollary}\label{cor:le:fincov:fincov}
  Let $\cE$ be a subhomogeneous (F) Banach bundle over compact Hausdorff $X$, conditionally f.t. and with $G_{\delta}$ $X_{\le d}$.

  $X$ has a finite closed cover $X=\bigcup_{i=1}^s Y_i$ with
  \begin{equation*}
    \forall i
    \quad
    \Gamma(\cE|_{Y_i})
    \cong
    \bigoplus_j^{\text{finite sum}}C_0(U_{ij})
    ,\quad
    U_{ij}\subseteq Y_i\text{ open $F_{\sigma}$}.
  \end{equation*}
\end{corollary}
\begin{proof}
  We can assume $\cE$ conditionally trivial by \Cref{le:fincov}, with fiber dimensions $d_0<\cdots <d_{k-1}$. $d_0$ sections $\sigma_j$ restricting to a basis in every fiber above $X_{d_0}$ can be extended to a closed neighborhood $U\supset X_{d_0}$. The restriction $\cE|_{\overline{X\setminus U}}$ has fewer fiber dimensions and can be taken care of by induction on $k$, so we can henceforth assume $U=X$. 

  Let $\cF\le \cE$ be the trivial rank-$d_0$ subbundle generated by $\{\sigma_j\}_j$. The quotient $\cQ:=\cE/\cF$ is trivial of respective rank $d_i-d_0$ over $X_{\cE=d_i}$. Focusing for the moment on $X_{\le d_1}$, one can always find $d_1-d_0$ sections $\sigma'_{j'}\in \Gamma_0\left(\cE|_{X_{\le d_1}}\right)$ generating $\cQ|_{X_{\le d_1}}$:
  \begin{itemize}[wide]
  \item choose $d_1-d_0$ sections in $\Gamma_0\left(\cQ|_{X_{\le d_1}}\right)$ forming a basis in every fiber above $X_{d_1}$ (this is possible \cite[\S 33, Exercise 4]{mnk}, the closed subset $X_{d_0}\subseteq X_{d_1}$ being $G_{\delta}$);

  \item extend those sections by 0 to all of $X_{\le d_1}$;

  \item lift to sections of $\cE$ (\Cref{res:quotnotf}\Cref{item:res:quotnotf:sectquot});

  \item and finally, subtract appropriate $C(X)$-linear combinations of the initial $d_0$ sections $\sigma_j$ to ensure vanishing along $X_{d_0}$.
  \end{itemize}
  Now repeat the procedure in the first paragraph of the proof: the $\sigma'_{j'}$ retain the basis property over a closed neighborhood $V\supseteq X_{\le d_1}$, which we may as well assume is all of $X$ by relegating $\overline{X\setminus V}$ to the induction hypothesis. But now we have a larger subbundle
  \begin{equation*}
    \cF\oplus \cF'\le \cE
    \text{ (equality over $X_{\le d_1}$)}
    ,\quad
    \cF:=\spn{\sigma_j}
    ,\quad
    \cF':=\spn_{\sigma'_{j'}}
  \end{equation*}
  with
  \begin{equation*}
    \Gamma(\cF)\cong C(X)^{d_0}
    \quad\text{and}\quad
    \Gamma(\cF')\cong C_0(X_{>d_0})^{d_1-d_0}.
  \end{equation*}
  It will be clear now how the recursion proceeds. 
\end{proof}

\section{Complements on tensor products and bundle categories}\label{se:tens} 

\Cref{pr:pullbclcov} relies implicitly on the {\it restriction} functor
\begin{equation}\label{eq:res2y}
  \tensor*[_{C(X)}]{\cat{Ban}}{}
  \ni
  M
  \xmapsto[\bullet|_{Y}]{\quad C(Y)\wo_{C(X)}-\ \cong\ C(Y)\otimes_{C(X)}-}
  M|_{Y}
  :=
  M/I_Y M  
  \in
  \tensor*[_{C(Y)}]{\cat{Ban}}{}
\end{equation}
where
\begin{itemize}[wide]
\item $Y\subseteq X$ is a closed embedding of compact Hausdorff spaces with
  \begin{equation*}
    I_{Y}
    :=
    \left\{f\in C(X)\ |\ f|_Y\equiv 0\right\}.
  \end{equation*}

\item One need not distinguish between the plain algebraic span $I_Y M$ and its closure by the Cohen factorization theorem \cite[Theorem 32.22]{hr-2}: $I_Y M$ is automatically closed in $M$, so that $M/I_YM$ is a Banach module over $C(Y)=C(X)/I_Y$ (cf. also remarks to the same effect in \cite[\S 1, last paragraph on p.434 and discussion immediately preceding Theorem 1.7]{zbMATH04134853}, for example).

\item And the preceding remark also justifies, in this context, the isomorphism between the {\it projective} module tensor product $\wo_{C(X)}$ \cite[\S III.3.8]{clm_ban-mod} and its purely algebraic counterpart.
\end{itemize}

\Cref{pr:pullbclcov} and \Cref{cor:injbijclcov} are meant to suggest that Banach $C(X)$-modules play well with the finite closed covers provided by \Cref{le:fincov}.

\begin{proposition}\label{pr:pullbclcov}
  Let $X=\bigcup_{i=0}^{n-1}X_i$ be a closed cover of a compact Hausdorff space, $M\in \tensor*[_{C(X)}]{\cat{Ban}}{}$, and set
  \begin{equation*}
    X_{\bf i}:=\bigcap_{i\in {\bf i}}X_i
    ,\quad
    {\bf i}\in\text{ powerset }2^{[n]} = 2^{\{0\cdots n-1\}}.
  \end{equation*}
  The canonical $\tensor*[_{C(X)}]{\cat{Ban}}{}$-morphism
    \begin{equation}\label{eq:canlim}
      M
      \xrightarrow{\quad\cat{can}\quad}
      L
      :=
      \varprojlim_{2^{[n]}}\left({\bf i}\xmapsto{\quad}M|_{X_{\bf i}}\right)
      :=
      \text{limit in }\tensor*[_{C(X)}]{\cat{Ban}}{}
    \end{equation}
    is an isomorphism in $\tensor*[_{C(X)}]{\cat{Ban}}{^{\infty}}$.
\end{proposition}
\begin{proof}
  Being left adjoint to scalar restriction along $C(X)\xrightarrowdbl{}C(Y)$, the functor \Cref{eq:res2y} is {\it cocontinuous} by, say, \cite[Theorem 4.5.3]{rhl_ct-ctxt} (i.e. \cite[\S 3.5]{rhl_ct-ctxt} it preserves colimits). In particular it preserves the cokernel $Q:=L/\ol{\im\cat{can}}$, so it will suffice to show that
  \begin{equation}\label{eq:bddbelow}
    \exists C>0
    \quad
    \forall m\in M
    \quad
    \|m\|\le C\|\cat{can}(m)\|:
  \end{equation}
  on its own this will prove the injectivity and image closure of $\cat{can}$, while applied to $Q$ in place of $M$ it will also prove that $Q$ is trivial and hence $\im\cat{can}\le L$ is dense.

  As for \Cref{eq:bddbelow}, it will be enough to take $C:=$the $n$ of the statement (the number of closed patches constituting the cover). To see this, recall first that
  \begin{equation*}
    \forall m\in M
    ,\quad
    \forall i
    ,\quad
    \left\|m|_{X_i}\right\|
    =
    \inf\left\{\|fm\|\ \big|\ X\xrightarrow{f\in C(X)}[0,1],\ f|_{\text{some nbhd of $X_i$}}\equiv 1\right\}
  \end{equation*}
  (\cite[Lemma 7.5]{gierz_bdls}, originally due to Varela in somewhat weaker form: \cite[Lemma 1.2]{zbMATH03432957}, \cite[Lemma 3.2]{zbMATH03442669}, \cite[Proposition 2.1]{dg_ban-bdl}). Choosing continuous
  \begin{equation*}
    X
    \xrightarrow{\quad f_i\quad}
    [0,1]
    ,\quad
    f_i|_{X_i}\equiv 1
  \end{equation*}
  with $\|f_i m\|$ very close (symbol: $\sim$) to $\left\|m|_{X_i}\right\|$, $\forall i$ for arbitrary $m\in M$, the function $f:=\sum_i f_i$ is invertible in $C(X)$ with inverse valued in $[0,1]$ and hence of norm $\le 1$. But then
  \begin{equation*}    
    \|m\|
    =
    \left\|f^{-1}\cdot fm\right\|
    \le
    \|fm\|
    \le
    \sum_i \left\|f_i m\right\|
    \sim
    \sum_i \left\|m|_{X_i}\right\|
    \le
    n\cdot \max_i\left\|m|_{X_i}\right\|
    =
    n\|\cat{can}(m)\|,
  \end{equation*}
  finishing the proof.
\end{proof}

\begin{remark}\label{re:notisometry}
  Naturally, there is no reason why \Cref{eq:canlim} would be an isomorphism in the smaller category $\tensor*[_{C(X)}]{\cat{Ban}}{}$, i.e. an isometry; {\it Hilbert} $C(X)$-modules \cite[Definition 15.1.5]{wo} $C(X)^n$ provide simple counterexamples for finite $X$ with at least two elements. 
\end{remark}

An immediate consequence:

\begin{corollary}\label{cor:injbijclcov}
  Given a finite closed cover $X=\bigcup_{i=0}^{n-1} X_i$ of a compact Hausdorff space, a morphism $M\xrightarrow{f}N$ in either $\tensor*[_{C(X)}]{\cat{Ban}}{}$ or $\tensor*[_{C(X)}]{\cat{Mod}}{}$ is injective (bijective) if and only if every restriction $f|_{X_i}$ is. 
\end{corollary}
\begin{proof}
  Indeed, we have linear and topological identifications \Cref{eq:canlim} for both $M$ and $N$, and limits respect both injectivity and bijectivity. 
\end{proof}

\begin{remarks}\label{res:notmono}
  \begin{enumerate}[(1),wide]
  \item \Cref{cor:injbijclcov} (or \Cref{pr:pullbclcov}) cannot hold for plain, purely algebraic modules, even for as well-behaved a Banach algebra as $C(X)$ with compact Hausdorff $X$: \Cref{ex:notmono} shows that the $\tensor*[_{C(X)}]{\cat{Mod}}{}$ version of \Cref{eq:canlim} may fail to be injective. It cannot vanish though, unless $M$ itself does: if $M=I_{X_i}M$ for all $i$ then $M=\left(\prod_i I_{X_i}\right)\cdot M=\{0\}$.

  \item The preceding observation is reminiscent of other mismatch phenomena whereby modules over Banach algebras are in some fashion better behaved analytically than algebraically. By \cite[Theorem VII.1.5]{hlmsk_homolog}, for instance, quotients
    \begin{equation}\label{eq:cw2cy}
      C(X)
      \xrightarrowdbl{\quad}
      C(Y)
      ,\quad
      Y\subseteq X\text{ embedding of compact Hausdorff spaces} 
    \end{equation}
    are {\it strictly flat} \cite[Definition VII.1.3]{hlmsk_homolog}: $C(Y)\wo_{C_X}-$ preserves the exactness of Banach-module complexes. Such quotients are nevertheless only very rarely {\it flat} in the ordinary sense (e.g. \cite[Definition 4.0]{lam_lec}) of the functor $C(Y)\otimes_{C(X)}-$ being exact: see \Cref{le:whenquotflat} below.
  \end{enumerate}  
\end{remarks}

\begin{example}\label{ex:notmono}
  Let $X:=[-1,1]$ with finite closed cover
  \begin{equation*}
    X=\left(I_-:=[-1,0]\right)\bigcup \left(I_+:=[0,1]\right)
  \end{equation*}
  and write
  \begin{equation*}
    f_+:=\chi_{I_+}\cdot\id|_X
    \quad\text{and}\quad
    f_-:=\chi_{I_-}\cdot\id|_X
  \end{equation*}
  ($\chi_{\bullet}$ denoting characteristic functions; so $f_+$ is constantly 0 on $I_-$ and the identity on the right-hand half $I_+$ of $X$); then set
  \begin{equation*}
    M_+:=C(X)/f_+\cdot I_0
    \quad\text{and}\quad
    M_-:=C(X)/f_-\cdot I_0.
  \end{equation*}
  The function $f_+$ itself does not belong to $f_+\cdot I_0$, so 
  \begin{equation*}
    m_+:=\text{image of $f_+\in C(X)$ through }
    C(X)
    \xrightarrowdbl{\quad}
    M_+
  \end{equation*}
  is non-zero, contained in $I_- M_+$ (for $f_+\in I_-$), and annihilated by $I_0$; the sign-mirror of this remark is also valid, producing a non-zero $m_-\in I_+ M_-$ and annihilated by $I_0$. Finally, take for $M$ the {\it pushout} \cite[\S 2.5]{brcx_hndbk-1} in $\tensor*[_{C(X)}]{\cat{Mod}}{}$ (plain modules) of
  \begin{equation*}
    \begin{tikzpicture}[>=stealth,auto,baseline=(current  bounding  box.center)]
      \path[anchor=base] 
      (0,0) node (l) {$M_-$}
      +(4,.5) node (u) {$\bC\cong C(X)/I_0$}
      +(8,0) node (r) {$M_+$}
      +(4,-.5) node (d) {$M$}
      ;
      \draw[left hook->] (u) to[bend right=6] node[pos=.5,auto,swap] {$\scriptstyle m_-\mapsfrom 1$} (l);
      \draw[right hook->] (u) to[bend left=6] node[pos=.5,auto] {$\scriptstyle 1\mapsto m_+$} (r);
      \draw[right hook->] (l) to[bend right=6] node[pos=.5,auto,swap] {$\scriptstyle $} (d);
      \draw[left hook->] (r) to[bend left=6] node[pos=.5,auto] {$\scriptstyle $} (d);
    \end{tikzpicture}
  \end{equation*}
  (the inward arrows are indeed injective \cite[Pushout theorem 2.54$^*$]{freyd_abcats}). The single (non-zero) element $m$ that is the common image of $m_{\pm}\in M_{\pm}$ belongs to both $I_{\pm} M$ and is thus annihilated by both restrictions to the two interval halves. 
\end{example}


\begin{lemma}\label{le:whenquotflat}
  Let $Y\subseteq X$ be an embedding of compact Hausdorff spaces. The following conditions are equivalent.
  \begin{enumerate}[(a),wide]
  \item\label{item:le:whenquotflat:flat} The quotient $C(X)\xrightarrowdbl{\pi}C(Y)$ is flat as a $C(X)$-module.

  \item\label{item:le:whenquotflat:idsum} For every $f\in I_Y:=\ker\pi$ we have
    \begin{equation}\label{eq:iyann}
      I_Y+\cat{Ann}_{C(X)}(f)=C(X)
      ,\quad
      \cat{Ann}_{C(X)}(f)
      :=
      \left\{g\in C(X)\ |\ gf=0\right\}.
    \end{equation}

  \item\label{item:le:whenquotflat:nbhd0} Every $f\in C(X)$ vanishing on $Y$ vanishes on a neighborhood of $Y$.
  \end{enumerate}
\end{lemma}
\begin{proof}
  The equivalence \Cref{item:le:whenquotflat:flat} $\Leftrightarrow$ \Cref{item:le:whenquotflat:idsum} is a direct translation to the present setup of \cite[Corollary 3]{zbMATH03340334} (itself a consequence of the {\it equational criterion} for flatness \cite[Theorem 4.24]{lam_lec}).

  \Cref{item:le:whenquotflat:idsum} $\Leftrightarrow$ \Cref{item:le:whenquotflat:nbhd0}, on the other hand, follows from
  \begin{equation*}
    \cat{Ann}_{C(X)}(f)
    =
    I_{\mathrm{supp}(f)}
    ,\quad
    \mathrm{supp}(f):=\overline{\left\{\text{non-zero locus of }f\right\}}:
  \end{equation*}
  \Cref{eq:iyann} reads
  \begin{equation*}
    I_{Y\cap\; \mathrm{supp}(f)}
    =
    I_Y+I_{\mathrm{supp}(f)}=C(X),
  \end{equation*}
  or just plain $Y\cap \mathrm{supp}(f)=\emptyset$. 
\end{proof}

The conditions of \Cref{le:whenquotflat} of course obtain if $Y\subseteq X$ happens to be clopen (for then $C(Y)\cong I_{X\setminus Y}$ is in fact a summand of $C(X)$ and hence also $C(X)$-projective), but not only then. The paradigmatic instance of this is perhaps the following. 

\begin{example}\label{ex:omegaintrv}
  Take for $Y\subseteq X$ the embedding
  \begin{equation*}
    \{\Omega\}
    \subset
    [0,\Omega]
    ,\quad
    \Omega:=\text{first uncountable ordinal}
  \end{equation*}
  of \cite[Example 43]{ss_countertop} (with $[0,\Omega]$ equipped with the {\it order topology} \cite[Example 43]{ss_countertop}). Every function vanishing at $\Omega$ vanishes on an entire neighborhood thereof by \cite[Example 43, item 12]{ss_countertop}. By \Cref{le:whenquotflat}, the quotient $C([0,\Omega])\xrightarrowdbl{\text{evaluation at $\Omega$}}\bC$ is flat. 
\end{example}


Returning to the tensor products of the present section's title, recall the bifunctor
\begin{equation*}
  \text{bundle pair }(\cE,\cF)\text{ over $X$}
  \xmapsto{\quad}
  \cE\wo \cF
  =
  \cE\wo_X\cF
\end{equation*}
of \cite[p.445]{zbMATH04134853} (where the symbol is $\astrosun_X$ instead; based on \cite[Theorem 2.1]{zbMATH03663842}). As the notation suggests, the fiber at $x\in X$ is the projective Banach tensor product $\cE_x\wo \cF_x$. The module (as opposed to bundle) side of the picture is as follows ($X$ being assumed compact Hausdorff throughout):
\begin{itemize}[wide]

\item There is an adjunction \cite[Theorems 2.5 and 2.6 and interlined discussion]{dg_ban-bdl}
  \begin{equation*}
    \begin{tikzpicture}[>=stealth,auto,baseline=(current  bounding  box.center)]
      \path[anchor=base] 
      (0,0) node (l) {$\text{(H) bundles on $X$}=:\cat{Bun}_X$}
      +(8,0) node (r) {$\tensor*[_{C(X)}]{\cat{Ban}}{}$}
      +(4,0) node (m) {$\top$}
      ;
      \draw[->] (l) to[bend left=16] node[pos=.5,auto] {$\scriptstyle \Gamma(\bullet)$} (r);
      \draw[->] (r) to[bend left=16] node[pos=.5,auto] {$\scriptstyle \cE_{\bullet}$} (l);
    \end{tikzpicture}
  \end{equation*}
  (the narrow end of `$\top$' pointing towards the left adjoint, per standard notation: \cite[Proposition 3.4.1]{brcx_hndbk-1}, \cite[Examples 19.4]{ahs}, etc.). By \cite[p.44, equation (2.1)]{dg_ban-bdl} the right adjoint $\Gamma(\bullet)$ is {\it fully faithful} or, equivalently \cite[Proposition 3.4.1]{brcx_hndbk-1}, the left-hand-based loop is the identity.
  
\item It follows that that fully-faithful global-section functor implements an equivalence \cite[Scholium 6.7]{hk_shv-bdl} between bundles on the one hand and the functor's essential image on the other: the full subcategory
  \begin{equation}\label{eq:lcban2ban}
    \tensor*[_{C(X)}]{\cat{lcBan}}{}
    \lhook\joinrel\xrightarrow{\quad}
    \tensor*[_{C(X)}]{\cat{Ban}}{}
  \end{equation}
   of {\it (locally) (C(X)-)convex} Banach $C(X)$-modules in the sense of \cite[p.40, pre Theorem 2.5]{dg_ban-bdl}, \cite[Definition 7.10]{gierz_bdls}, \cite[\S 6.1]{hk_shv-bdl}, and so on:
  \begin{equation*}
    \|fm+(1-f)m'\|\le \max\left(\|m\|,\ \|m'\|\right)
    ,\quad
    \forall X\xrightarrow[\quad\text{continuous}\quad]{f}[0,1]
    \text{ and }
    m,m'\in M.
  \end{equation*}

\item By \cite[Theorem 1.1]{zbMATH04134853} the (full) subcategory \Cref{eq:lcban2ban} is {\it reflective} in the sense \cite[Definition 3.5.2]{brcx_hndbk-1} that the embedding has a left adjoint
  \begin{equation*}
    \tensor*[_{C(X)}]{\cat{Ban}}{}
    \xrightarrow{\quad\text{{\it Gelfand functor }$\cat{lc}$}\quad}
    \tensor*[_{C(X)}]{\cat{lcBan}}{}
    \quad
    \left(\text{the $\cG$ of \cite[p.435]{zbMATH04134853}}\right).
  \end{equation*}
  The functor $\cat{lc}$ implements a kind of ``universal convexification'' $M\to \cat{lc}(M)$ for an arbitrary Banach $C(X)$-module. 
\item And finally, the bundle / convex-module correspondence identifies $\cE\wo_X \cF$ with
  \begin{equation}\label{eq:lcwo}
    \Gamma(\cE\wo_X\cF)
    \cong
    \Gamma(\cE)\tensor[_{lc}]{\wo}{_{C(X)}}\Gamma(\cF)
    :=
    \cat{lc}\left(\Gamma(\cE)\tensor[]{\wo}{_{C(X)}}\Gamma(\cF)\right).
  \end{equation}
\end{itemize}

The following statement is phrased along the lines of \Cref{le:fincov}, with the same intention of rendering it independent of prior work; by \cite[Theorem 1.11]{2405.14518v1} the hypothesis could be packaged more economically as the requirement that the bundles be topologically f.g.

\begin{theorem}\label{th:tensidentify}
  Let $\cE$ and $\cF$ be two subhomogeneous (F) Banach bundles over compact Hausdorff $X$, conditionally f.t. and with $G_{\delta}$ $X_{\le d}$.

  The canonical $\tensor*[_{C(X)}]{\cat{Ban}}{}$-morphism  
  \begin{equation*}
    \Gamma(\cE)\wo_{C(X)}\Gamma(\cF)
    \xrightarrow{\quad}
    \Gamma\left(\cE\wo_X\cF\right)
  \end{equation*}
  is bijective. 
\end{theorem}
\begin{proof}
  \Cref{le:fincov} and \Cref{cor:injbijclcov} reduce the problem to conditionally trivial bundles, in which case the claim will be easy to verify. Consider the strata \Cref{eq:strata} of $\cE$. By \Cref{cor:le:fincov:fincov} (and its proof), upon perhaps refining the finite closed cover, we can assume $\cE$ is a direct sum of
  \begin{itemize}[wide]
  \item a trivial rank-$d_0$ bundle;

  \item and a trivial rank-$(d_1-d_0)$ bundle over $X_{\cE>d_0}$ (vanishing along $X_{\cE=d_0}=X_{\cE\le d_0}$)

  \item and so on, up to a trivial rank-$(d_{k-1}-d_{k-2})$ bundle over $X_{\cE>d_{k-2}}$ vanishing along $X_{\cE\le d_{k-2}}$. 
  \end{itemize}
  The same goes for $\cF$, so all relevant $C(X)$-modules are closed ideals thereof. For these the projective tensor product is nothing but the product:
  \begin{equation*}
    I\wo_{C(X)}J
    \xrightarrow[\quad\cong\quad]{\quad\text{canonical map}\quad}
    IJ
  \end{equation*}
  by (one version of) Cohen factorization again (e.g. \cite[Theorem III.3.12]{hlmsk_homolog}), and the conclusion follows. 
\end{proof}

In particular, linking back to finite generation:

\begin{corollary}\label{cor:eotffg}
  If $\cE$ and $\cF$ are (topologically) f.g. (F) bundles over a compact Hausdorff space $X$ so, respectively, is $\cE\wo_X\cF$.
\end{corollary}
\begin{proof}
  That topological finite generation transports over from $\Gamma(\cE)$ and $\Gamma(\cF)$ to $\Gamma(\cE)\wo_{C(X)}\Gamma(\cF)$ is immediate, hence the topological branch of the claim. As for the plain (purely algebraic) side, recall \cite[Theorem 1.11]{2405.14518v1} that topological finite generation also implies the surjectivity of
  \begin{equation}\label{eq:alg2top}
    (\text{algebraic tensor product})
    \quad
    \Gamma(\cE)\otimes_{C(X)}\Gamma(\cF)
    \xrightarrow{\quad}
    \Gamma(\cE)\wo_{C(X)}\Gamma(\cF),
  \end{equation}
  so that if the two individual tensorands are f.g. so is their projective tensor product. Both versions of the claim now follow from \Cref{th:tensidentify}.
\end{proof}

For compact Hausdorff $X$ we have a number of {\it symmetric monoidal} categories \cite[Definitions 6.1.1 and 6.1.2]{brcx_hndbk-2}:
\begin{itemize}[wide]
\item $\tensor*[]{\cat{Bun}}{_X}$, of (H) Banach bundles, with $\wo_X$ (and the trivial rank-1 bundle as {\it monoidal unit});

\item the equivalent category $\tensor*[_{C(X)}]{\cat{lcBan}}{}$ of convex Banach $C(X)$-modules with the tensor product $\tensor[_{lc}]{\wo}{_{C(X)}}$ of \Cref{eq:lcwo}; the equivalence
  \begin{equation}\label{eq:bun2mod}
    \tensor*[]{\cat{Bun}}{_X^{\bullet}}
    \xrightarrow[\quad\simeq\quad]{\Gamma(-)}
    \tensor*[_{C(X)}]{\cat{lcBan}}{^{\bullet}}
    ,\quad
    \bullet\in\{\text{blank},\ \infty\}
  \end{equation}
  of \cite[Theorem 2.6]{dg_ban-bdl} is by the very definition of $\tensor[_{lc}]{\wo}{_{C(X)}}$ symmetric monoidal (as a functor, i.e. {\it braided monoidal} in the sense of \cite[\S XI.2]{mcl_2e}).
\end{itemize}

In addition, a reformulation of \Cref{cor:eotffg} gives

\begin{corollary}\label{cor:symmmon}
  For compact Hausdorff $X$ the full subcategory
  \begin{equation*}
    \tensor*[^{\cat{(t)fg}}^{(F)}]{\cat{Bun}}{_X}
    \quad\subset\quad
    \tensor*[]{\cat{Bun}}{_X}
  \end{equation*}
  of (topologically) f.g. (F) Banach bundles is closed under $\wo_X$, and hence inherits a symmetric monoidal structure from the larger ambient category.  \qedhere
\end{corollary}

\begin{remark}\label{re:indepofth}
  We know from \Cref{th:loctrivfg} that just plain f.g. bundles are locally trivial, but \Cref{cor:symmmon} does not rely on that earlier result. 
\end{remark}

Like any symmetric monoidal category linear over a characteristic-0 field, $\tensor*[]{\cat{Bun}}{_X}$ comes equipped with {\it Schur (endo)functors} \cite[\S 6.1]{fh_rep-th}
\begin{equation*}
  \cE
  \xmapsto{\quad}
  \bS_{\lambda}\cE
  :=
  \left(V_{\lambda}\otimes \cE^{\otimes n}\right)^{S_n}
  :=
  \text{$S_n$-invariants},
\end{equation*}
where
\begin{itemize}[wide]
\item $S_n$ is the $n$-symbol symmetric group for $n\in \bZ_{>0}$;

\item $\lambda=(\lambda_1\ge \lambda_2\ge \cdots)$ is a {\it partition} \cite[\S 4.1]{fh_rep-th} of $n$ (i.e. $\sum \lambda_i=n$);

\item $V_{\lambda}$ is the corresponding irreducible $S_n$-representation, per the usual \cite[Theorem 4.3]{fh_rep-th} classification;

\item and $S_n$ operates on $\cE^{\otimes n}$ ($n^{th}$ tensor power with respect to $\wo_X$) by permuting tensorands and thence on $V_{\lambda}\otimes \cE^{\otimes n}$ in the obvious fashion. 
\end{itemize}

\Cref{cor:symmmon} now also implies

\begin{corollary}\label{cor:schurfunct}
  If $\cE\xrightarrowdbl{}X$ is a (topologically) f.g. (F) bundle over compact Hausdorff $X$ so, respectively, are the Schur images $\bS_{\lambda}\cE$ for partitions $\lambda$.  \qedhere
\end{corollary}

\begin{remark}\label{re:altpf}
  \Cref{cor:schurfunct} applies in particular to {\it symmetric} and {\it exterior powers}:
  \begin{equation*}
    S^n\cE:=\bS_{(n)}\cE
    \quad\text{and}\quad
    \bigwedge^n\cE:=\bS_{(1,1,\cdots)}\cE
    \quad\text{respectively}
    \quad\text{\cite[(6.1) and (6.2)]{fh_rep-th}}.
  \end{equation*}
  This observation affords an alternative approach to part of the proof of \Cref{th:loctrivfg} (hence the relevance of having the present material available independently of that theorem: \Cref{re:indepofth}): the transition between steps \Cref{item:th:loctrivfg:2str} and \Cref{item:th:loctrivfg:01str} in that proof can be effected by the substitution $\cE\mapsto \bigwedge^{d_1}\cE$ for $\cE$ with fiber dimensions $d_0<d_1$. The exterior power is again f.g. by \Cref{cor:schurfunct}, and furthermore
  \begin{equation*}
    X_{\bigwedge^{d_1}\cE=i} = X_{\cE=i}
    ,\quad
    i=0,1.
  \end{equation*}
  For the purpose of proving these sets clopen in $X$, then, working with one bundle is as effective as working with the other. 
\end{remark}

The surjectivity of the map \Cref{eq:alg2top} was crucial to the proof of \Cref{cor:eotffg}. More is true: \Cref{th:algeqtop} confirms the affirmative answer to the obvious question. Recall \cite[Problem 17I]{wil_top} that {\it $\sigma$-compact} spaces are those expressible as countable unions of compact subspaces. 

\begin{theorem}\label{th:algeqtop}
  Let $X$ be a locally compact $\sigma$-compact Hausdorff space and $\cE$, $\cF$ two subhomogeneous, conditionally f.t. (F) bundles with $G_{\delta}$ $X_{\le d}$.

  The morphisms
  \begin{equation}\label{eq:allmor}
    (\text{algebraic tensor product})
    \quad
    \Gamma_0(\cE)\otimes_{C_0(X)}\Gamma_0(\cF)
    \xrightarrow{\quad}
    \Gamma_0(\cE)\wo_{C_0(X)}\Gamma_0(\cF)
    \xrightarrow[\cong]{\quad}
    \Gamma_0(\cE\wo_X \cF)
  \end{equation}
  are then all bijective.
\end{theorem}

\begin{remark}\label{re:th:algeqtop:clc}
  The statement refers to locally compact spaces in order to facilitate induction on the number of strata. Note however we can pass back and forth between the locally compact and compact versions of the problem: the former setup is plainly broader, but it can also be reduced to compact spaces by
  \begin{itemize}[wide]
  \item substituting the {\it one-point compactification} \cite[Definition 19.2]{wil_top} $X^+\supseteq X$ ($X$ itself if already compact, $X\sqcup\{\text{an extra point}\}$ otherwise) for $X$;

  \item and extending the bundles $\cE\mapsto \cE^+$, $\cF\mapsto \cF^+$ by 0 at the extraneous point;

  \item whereupon $\Gamma_0(\cE)=\Gamma(\cE^+)$, etc.
  \end{itemize}
  By the same token, \Cref{cor:eotffg} is valid (and will be taken for granted) for locally compact $\sigma$-compact base spaces: this is what justifies the `$\cong$' symbol in the statement of \Cref{th:algeqtop}. 
\end{remark}

\pf{th:algeqtop}
\begin{th:algeqtop}
  Surjectivity is settled by (the proof of) \Cref{cor:eotffg}, via \Cref{re:th:algeqtop:clc} (which transports the compact-base version over to the present seemingly-broader setup). It is thus {\it in}jectivity that requires more attention.

  \begin{enumerate}[(I),wide]
  \item\label{item:th:algeqtop:2triv} {\bf Locally trivial $\cE$ and $\cF$.} In that case the claim follows from Swan's theorem \cite[Theorem 2]{zbMATH03179258}, classifying section spaces as precisely the projective f.g. modules over $C(X)$ for compact Hausdorff $X$. That result goes through under the present assumptions, with $C_0(X)$ and $\Gamma_0(\cE)$ in place of $C(X)$ and $\Gamma(\cE)$ respectively:
    \begin{itemize}[wide]
    \item \cite[\S 1]{zbMATH03179258} requires only paracompactness, valid \cite[Theorem 20.12(a)]{wil_top} for $\sigma$-compact locally compact Hausdorff spaces;

    \item \cite[\S 2]{zbMATH03179258} requires normality, a consequence of paracompactness \cite[Theorem 20.10]{wil_top};

    \item and finally, compactness itself is only needed in \cite[\S 3, Lemma 5]{zbMATH03179258} to conclude that $X$ is can be covered with finitely many open patches where local sections form a base in every fiber; the finite-type assumption provides that ingredient.
    \end{itemize}

  \item\label{item:th:algeqtop:1triv} {\bf Locally trivial $\cE$ or $\cF$.} Assume $\cF$ locally trivial, to fix ideas, and suppose $\cE$ has $k$ fiber dimensions $d_0<d_1<\cdots<d_{k-1}$. Consider also an element
    \begin{equation}\label{eq:annu}
      u
      =
      \sum_i \sigma_i\otimes \tau_i
      \in
      \Gamma_0(\cE)\otimes_{C_0(X)}\Gamma_0(\cF)
    \end{equation}
    annihilated by \Cref{eq:allmor}. Step \Cref{item:th:algeqtop:2triv} applied to $X_{\cE=d_0}$ shows that $u$ belongs to
    \begin{equation*}
      \begin{aligned}
        I_{Y}\cdot \left(\Gamma_0(\cE)\otimes_{C_0(X)}\Gamma_0(\cF)\right)
        &=\left(I_Y\cdot \Gamma_0(\cE)\right)\otimes_{C_0(X)}\left(I_Y\cdot \Gamma_0(\cE)\right)
        \\
        &\subseteq \Gamma_0(\cE)\otimes_{C_0(X)}\Gamma_0(\cF)
      \end{aligned}      
    \end{equation*}
    for $ Y:=X_{\cE=d_0}$, equality following once more from Cohen factorization \cite[Theorem 32.22]{hr-2} (every element of $I_Y$ is a product of two such). $u$ is thus the image of an element of
    \begin{equation*}
      \Gamma_0(\cE|_{U})\otimes_{C_0(U)} \Gamma_0(\cF|_{U})
      ,\quad
      U:=X\setminus Y = X_{\cE>d_0},
    \end{equation*}
    and the problem reduces fewer fiber dimensions $d_1<\cdots<d_{k-1}$ by restricting everything in sight to $U$. Induction on $k$ and step \Cref{item:th:algeqtop:2triv} finish the argument. 

  \item\label{item:th:algeqtop:gen} {\bf The general case.} The preceding argument recycles: the tensorands $\sigma_i$ and $\tau_i$ of an element \Cref{eq:annu} in the kernel of \Cref{eq:allmor} can be assumed to vanish on the lowest stratum $X_{\cE=d_0}$ attached to $\cE$ by the previous step, and we can then proceed by induction on the {\it total} number $k+\ell$ of fiber dimensions
    \begin{equation*}
      d_0<\cdots<d_{k-1}     
      \text{ for $\cE$ and }
      d'_0<\cdots<d'_{k-1}
      \text{ for $\cF$}. 
    \end{equation*}
  \end{enumerate}
  This concludes the proof. 
\end{th:algeqtop}

\Cref{th:algeqtop} is one manifestation of agreement between algebraic and analytic structure for section modules of topologically f.g. bundles. That agreement goes further, with such modules exhibiting the type of {\it automatic-continuity} behavior that has seen much attention in the Banach-algebra literature (e.g. \cite[Chapter 5]{dales_autocont} and its numerous references).

\begin{theorem}\label{th:autocont}
  Let $\cE$ be a subhomogeneous (F) Banach bundle over compact Hausdorff $X$, conditionally f.t. and with $G_{\delta}$ $X_{\le d}$. 

  For any Banach $C(X)$-module $E$ (possibly degenerate) a $C(X)$-module morphism $\Gamma(\cE)\to E$ is automatically continuous. 
\end{theorem}
\begin{proof}
  \Cref{cor:le:fincov:fincov} and \Cref{pr:pullbclcov} reduce the problem to bundles $\cE$ with section spaces of the form $C_0(U)$ for open $U\subseteq X$, and for these (purely algebraic) $C(X)$-module morphisms $C_0(U)\to E$ are indeed automatically continuous by \cite[Corollary 2.9.30(ix)]{dales_autocont} applied to the ideal $C_0(U)\trianglelefteq C(X)$.
\end{proof}

In particular, when $E$ too is of the form $\Gamma(\cF)$ for $\cF$ as in the statement (i.e. \cite[Theorem 1.11]{2405.14518v1} topologically f.g.):

\begin{corollary}\label{cor:tfgfullinmod}
  \begin{enumerate}[(1),wide]
  \item\label{item:cor:tfgfullinmod:both} The faithful functor
    \begin{equation*}
      \tensor*[]{\cat{Bun}}{_X^{\infty}}
      \xrightarrow{\quad\Gamma(\bullet)\quad}
      \tensor*[_{C(X)}]{\cat{Mod}}{}
    \end{equation*}
    restricts on either category $\tensor*[^{\cat{(t)fg}}^{(F)}]{\cat{Bun}}{_X^{\infty}}$ (with `$\cat{t}$' present or not) to a symmetric monoidal equivalence onto a {\it full} category of $C(X)$-modules.

  \item\label{item:cor:tfgfullinmod:justproj} More precisely, the preceding item identifies $\tensor*[^{\cat{fg}}^{(F)}]{\cat{Bun}}{_X^{\infty}}$ with the category of projective f.g. $C(X)$-modules.  \qedhere
  \end{enumerate}
\end{corollary}

\begin{remarks}\label{res:innerhom}
  \begin{enumerate}[(1),wide]
  \item To further expand on the category-theoretic picture, recall the Banach-module version \cite[\S III.3.9, 4)]{clm_ban-mod} of the {\it hom-tensor adjunction} (e.g. \cite[\S 10.5, Theorem 43]{df_3e}):
    \begin{equation*}
      \tensor*[]{\cat{Ban}}{_B^{\infty}}(E\wo_A F,G)
      \quad
      \cong
      \quad
      \tensor*[]{\cat{Ban}}{_A^{\infty}}\left(E,\tensor*[]{\cat{Ban}}{_B^{\infty}}(F,G)\right)
      \quad\text{in $\cat{Ban}$ (so isometrically)}
    \end{equation*}
    functorially in $E \in \tensor[]{\cat{Ban}}{_A}$, $F \in \tensor[_A]{\cat{Ban}}{_B}$ and $G \in \tensor[]{\cat{Ban}}{_B}$. In particular, non-degenerate Banach modules over a commutative Banach algebra form a {\it closed} symmetric monoidal category under $\wo_A$ in the sense of \cite[Definition 6.1.3]{brcx_hndbk-2} (or \cite[\S VII.7]{mcl_2e}, say): there is an {\it internal (or inner) hom} \cite[Definition 7.9.2]{egno} (bi)functor
    \begin{equation*}
      \left(\tensor*[_A]{\cat{Ban}}{}\right)^{\circ}  
      \times
      \tensor*[_A]{\cat{Ban}}{}
      \ni
      (F,G)
      \xmapsto{\quad}
      [F,G]
      :=
      \tensor*[_A]{\cat{Ban}}{^{\infty}}(F,G)
      \in
      \tensor*[_A]{\cat{Ban}}{}
    \end{equation*}
    with $[F,-]$ right adjoint to $-\wo_A F$ (with $\cC^{\circ}$ denoting the opposite category). $[F,G]$ is automatically convex over $A:=C(X)$ if $F$ and $G$ are \cite[Proposition 1.14]{zbMATH04134853}, so the same internal homs provide monoidal closure for
    \begin{equation*}
      \left(\tensor*[_{C(X)}]{\cat{lcBan}}{},\ \tensor[_{lc}]{\wo}{_{C(X)}}\right)
      \quad
      \cong
      \quad
      \left(\tensor*[]{\cat{Bun}}{_X},\ \wo_X\right).
    \end{equation*}
    The sheaf corresponding to $[F,G]\in \tensor*[_A]{\cat{Ban}}{}$ through the equivalence of \cite[Scholium 6.7]{hk_shv-bdl} is the $\cat{Ban}$ analogue of the sheaf hom $\mathcal{H}om(\cE,\cF)$ of \cite[\S I.5, pp.20-21]{bred_shf_2e_1997}. 
    
  \item Let $A$ be a unital commutative ring. A unital $A$-module $M$ is projective f.g. precisely \cite[Example 3.3.11]{par_qg-ncg} when it is {\it rigid} \cite[Definition 3.3.1]{par_qg-ncg} in the symmetric monoidal category $(\tensor*[_A]{\cat{Mod}}{},\ \otimes_A)$: inner homs $[M,-]$ are of the form
    \begin{equation*}
      [M,-]\cong (-)\otimes M^*\text{ in }\tensor*[_A]{\cat{Mod}}{}
      ,\quad
      M^*:=[M,A] = \tensor*[_A]{\cat{Mod}}{}(M,A).
    \end{equation*}
    For that reason, the identification of \Cref{cor:tfgfullinmod}\Cref{item:cor:tfgfullinmod:justproj} between $\tensor*[^{\cat{fg}}^{(F)}]{\cat{Bun}}{_X^{\infty}}$ and the category of projective f.g. $C(X)$-modules shows that the former category is closed (in $\left(\tensor*[]{\cat{Bun}}{_X},\ \wo_X\right)$) under inner homs:
    \begin{equation*}
      [\cE,\cF]
      \cong
      \cF\wo_X\cE^*
      ,\quad
      \forall \cF\in \tensor*[]{\cat{Bun}}{_X}
      \quad\text{and}\quad
      \forall \cE\in \tensor*[^{\cat{fg}}^{(F)}]{\cat{Bun}}{_X},
    \end{equation*}
    with $\cE^*$ the inner hom $[\cE,\text{trivial rank-1 bundle}]$. 

  \item The preceding point notwithstanding, the larger category $\tensor*[^{\cat{tfg}}^{(F)}]{\cat{Bun}}{_X}\supseteq \tensor*[^{\cat{fg}}^{(F)}]{\cat{Bun}}{_X}$ is {\it not}, generally, closed in $\tensor*[]{\cat{Bun}}{_X}$ under inner homs.

    Consider an embedding $Y\subseteq X$ of compact Hausdorff spaces and set $U:=X\setminus Y$. $C(X)$-module morphisms $C_0(U)\to C(X)$ are automatically continuous \cite[Corollary 2.9.30(ix)]{dales_autocont} and take values in $C_0(U)$ because every element of the domain $C_0(U)$ is a product of two such (Cohen factorization \cite[Theorem 32.22]{hr-2}), so they are {\it multipliers} \cite[\S 2.2]{wo} of $C_0(U)$. We thus \cite[Examples 2.2.4]{wo} have an isomorphism 
    \begin{equation*}
      C(\beta U)
      \ni
      \left(U\xrightarrow[\text{bounded cont.}]{f}\bC\right)
      \xmapsto[\quad\cong\quad]{\quad}
      \left(g\mapsto gf\right)
      \in
      \tensor*[_{C(X)}]{\cat{Mod}}{}(C_0(U), C(X))
    \end{equation*}
    for the {\it Stone-\v{C}ech compactification} \cite[\S 1.5]{ddls_ban-sp-cont-fn} $\beta U\supseteq U$. The (convex Banach) $C(X)$-module $\tensor*[_{C(X)}]{\cat{Mod}}{}(C_0(U), C(X))$, then, corresponds via \Cref{eq:bun2mod} to a bundle $\cE\xrightarrowdbl{}X$ whose fiber over $x\in X$ is $C(\psi^{-1}(x))$ for
    \begin{equation*}
      \begin{tikzpicture}[>=stealth,auto,baseline=(current  bounding  box.center)]
        \path[anchor=base] 
        (0,0) node (l) {$U$}
        +(2,.5) node (u) {$\beta U$}
        +(4,0) node (r) {$X$}
        ;
        \draw[right hook->] (l) to[bend left=6] node[pos=.5,auto] {$\scriptstyle $} (u);
        \draw[->] (u) to[bend left=6] node[pos=.5,auto] {$\scriptstyle \psi$} (r);
        \draw[right hook->] (l) to[bend right=6] node[pos=.5,auto,swap] {$\scriptstyle $} (r);
      \end{tikzpicture}
    \end{equation*}
    $\cE|_U$ is thus trivial of rank 1, while the fibers over $x\in Y=X\setminus U$ can certainly be larger than 1-dimensional, generally (e.g. $Y\subseteq X = \{0\}\subset [0,1]$). In particular,  (the bundle corresponding to) $\tensor*[_{C(X)}]{\cat{Mod}}{}(C_0(U), C(X))$ need not be continuous even when $\cE$ and $\cF$ are. 
  \end{enumerate}  
\end{remarks}




\addcontentsline{toc}{section}{References}

\Addresses

\end{document}